\documentclass[11pt,a4paper,twoside]{article}
\usepackage{amsmath,amsthm,amssymb,amsfonts,amsbsy}

\usepackage{color, colortbl} 
\usepackage{float} 
\usepackage[margin=2.5cm]{geometry} 
\usepackage{cite}
\usepackage{authblk} 

\theoremstyle{definition}

\theoremstyle{plain}
\newtheorem{thm}[subsection]{Theorem}
\newtheorem{lem}[subsection]{Lemma}

\newcommand{\beq}{\begin{eqnarray}}
\newcommand{\eeq}{\end{eqnarray}}

\newcommand{\beqs}{\begin{eqnarray*}}
\newcommand{\eeqs}{\end{eqnarray*}}

\allowdisplaybreaks

\usepackage{mathtools}

\title{On the zeroth-order general Randi\'{c} index, variable sum exdeg index and trees having vertices with prescribed degree}

\author{
   \large \bf Sohaib Khalid, Akbar Ali\footnote{Corresponding author.}
}

\affil{ \normalsize
    { Department of Mathematics,\\  University of Management \& Technology, Sialkot-Pakistan}
    \\E-mail: {\tt sohaib4uus@gmail.com, akbarali.maths@gmail.com}
}

\begin{document}

\maketitle

\begin{abstract}

The zeroth-order general Randi\'{c} index (usually denoted by $R_{\alpha}^{0}$) and variable sum exdeg index (denoted by $SEI_{a}$) of a graph $G$ are defined as $R_{\alpha}^{0}(G)= \sum_{v\in V(G)} (d_{v})^{\alpha}$ and $SEI_{a}(G)= \sum_{v\in V(G)}d_{v}a^{d_{v}}$ where $d_{v}$ is degree of the vertex $v\in V(G)$, $a$ is a positive real number different from 1 and $\alpha$ is a real number other than $0$ and $1$. A segment of a tree is a path $P$, whose terminal vertices are branching or pendent, and all non-terminal vertices (if exist) of $P$ have degree 2. For $n\ge6$, let $\mathbb{PT}_{n,n_1}$, $\mathbb{ST}_{n,k}$, $\mathbb{BT}_{n,b}$ be the collections of all $n$-vertex trees having $n_1$ pendent vertices, $k$ segments, $b$ branching vertices, respectively. In this paper, all the trees with extremum (maximum and minimum) zeroth-order general Randi\'{c} index and variable sum exdeg index are determined from the collections $\mathbb{PT}_{n,n_1}$, $\mathbb{ST}_{n,k}$, $\mathbb{BT}_{n,b}$. The obtained extremal trees for the collection $\mathbb{ST}_{n,k}$ are also extremal trees for the collection of all $n$-vertex trees having fixed number of vertices with degree 2 (because it is already known that the number of segments of a tree $T$ can be determined from the number of vertices of $T$ with degree 2 and vise versa).

\end{abstract}
%
%
%
%
%
%
%
%
\section[Introduction]{Introduction}

Let $G=(V(G),E(G))$ be a finite and simple graph, where $V(G)$ and $E(G)$ are the nonempty sets, known as vertex set and edge set respectively. For a vertex $v\in V(G)$, degree of $v$ is denoted by $d_{v}$ and is defined as the number of vertices adjacent to $v$. Undefined terminologies and notations can be found in \cite{1,3}.

``A molecular descriptor is the final result of a logical and mathematical procedure which transforms chemical information encoded within a symbolic representation of a molecule into an useful number or the result of some standardized experiment'' \cite{5}.
A topological index is a type of molecular descriptor based on the molecular graph of chemical compounds \cite{7}. In graph theoretic words, topological indices are numerical quantities which are invariant under graph isomorphism \cite{8}. The Randi\'{c} index \cite{14} (devised in 1975 for measuring the branching of molecules) and first Zagreb index \cite{19} (appeared in 1972 within the study of total $\pi$-electron energy of molecules) are among the most studied topological indices \cite{15}. Kier and Hall \cite{16} proposed the zeroth order Randi\'{c} index. In 2005, general first Zagreb index (also known as the zeroth-order general Randi\'{c} index) was introduced by Li and Zheng \cite{9}. The zeroth-order general Randi\'{c} index is denoted by $R_{\alpha}^{0}$ and is defined as:
\[R_{\alpha}^{0}(G)= \sum_{v\in V(G)} (d_{v})^{ \alpha},\]
where $\alpha$ is a real number other than 0 and 1. Indeed, $R_{\alpha}^{0}$ reduces to first Zagreb index and zeroth-order Randi\'{c} index for $\alpha=2$ and $\alpha=-\frac{1}{2}$, respectively. The topological index $R_{\alpha}^{0}$ has attracted a considerable attention from mathematicians, for example see the papers \cite{4,17,Su14,20,Qiao10,Pav09,Pan11}, particularly the recent ones \cite{Chen17,Su16,Su16b,18} and related reference listed therein.

Variable sum exdeg index, introduced by Vuki\v{c}evi\'{c} \cite{10} in 2011, is denoted by $SEI_{a}$ and is defined as:
\[SEI_{a}(G)= \sum_{v\in V(G)}d_{v}a^{d_{v}},\]
where $a$ is any positive real number such that $a\neq1$. The topological index $SEI_{a}$ is very well correlated with octanol-water partition coefficient of octane isomers \cite{10}. Detail about the chemical applicability and mathematical properties of this index can be found in the references \cite{10, Ali-S,Yarahmadi15,ga-egrveivmi-17,v-ukicevic-11b}.

A vertex having degree 1 is called pendent vertex and a vertex which have degree greater than 2 is named as branching vertex. A segment of a tree is a path subtree $P$, whose terminal vertices are branching or pendent, and all non-terminal vertices (if exist) of $P$ have degree 2. The main purpose the present paper is to solve the problem of determining all the trees with extremum (maximum and minimum) zeroth-order general Randi\'{c} index and variable sum exdeg index from the collection of all $n$-vertex trees having fixed (i) pendent vertices (ii) segments (iii) branching vertices, is solved. The number of segments of a tree $T$ can be determined from the number of vertices of $T$ with degree 2 and vise versa \cite{12}. Hence, the obtained extremal trees for the collection $\mathbb{ST}_{n,k}$ are also extremal trees for the collection of all $n$-vertex trees having fixed number of vertices with degree 2.

Let $G'$ be a graph obtained from another graph $G$ by applying some graph transformation such that $V(G)=V(G')$. Throughout the paper, whenever such two graphs are under discussion, by the vertex degree $d_u$ we always mean degree of the vertex $u$ in $G$.

\section{Zeroth-order general Randi\'{c} index, variable sum exdeg index and pendent vertices of trees}

Denote by $n_i(G)$ (or simply by $n_i$) the number of vertices of graph $G$ having degree $i$. For $n\ge6$, let $\mathbb{PT}_{n,n_1}$ be the collection of all $n$-vertex trees with $n_1$ pendent vertices. Clearly, $2\le n_1 \le n-1$. Both the collections $\mathbb{PT}_{n,2}$ and $\mathbb{PT}_{n,n-1}$ contain only one graph, namely, the path graph $P_n$ and star graph $S_n$, respectively. Thereby, in order to make the extremal problem well defined we always take $3\le n_1\le n-2$.

The trees with extremum $SEI_{a}$ values from the collection $\mathbb{PT}_{n,n_1}$ have already been determined in \cite{v-ukicevic-11b} for $a>1$. Thereby, in this section, we solve this problem concerning $SEI_{a}$ for $0<a<1$, which gives a partial solution of a problem posed in \cite{v-ukicevic-11b}.

\begin{lem}\label{lem-p1}
If $T\in\mathbb{PT}_{n,n_1}$ contains more than one branching vertex then there exist a tree $T'\in\mathbb{PT}_{n,n_1}$ such that $SEI_{a}(T) >  SEI_{a}(T')$ for $0<a<1$ and
\[
R_{\alpha}^{0}(T)
\begin{cases}
<       R_{\alpha}^{0}(T') & \text{if $\alpha<0$ or $\alpha>1$,}\\
>       R_{\alpha}^{0}(T') & \text{if $0<\alpha<1$.}
\end{cases}
\]

\end{lem}

\begin{proof}

Let $u,v\in V(T)$ be branching vertices such that $d_u\ge d_v$. Let $w$ be the neighbor of $v$ which does not lie on the unique $u-v$ path. Take $T' = T - vw + uw$ then Lagrange's mean value theorem guaranties the existence of numbers $\Theta_{1}$, $\Theta_{2}$ such that $d_v -1 < \Theta_{1} < d_v \le d_u  < \Theta_{2} < d_u +1$ and
\begin{eqnarray}\label{Eq-p1}
SEI_{a}(T)- SEI_{a}(T')&=&d_v a^{d_v} - (d_v - 1 ) a^{d_v -1} - [(d_u +1) a^{d_u +1} - d_u  a^{d_u}]\nonumber\\
&=& a^{\Theta_{1}}(1+\Theta_{1}\ln a)-a^{\Theta_{2}}(1+\Theta_{2}\ln a)
\end{eqnarray}
From the inequalities $\Theta_{1} < \Theta_{2}$ and $0< a< 1$, it follows that
$$a^{\Theta_{1}}(1+\Theta_{1}\ln a)>a^{\Theta_{1}}(1+\Theta_{2}\ln a)> a^{\Theta_{2}}(1+\Theta_{2}\ln a),
$$
which together with Equation (\ref{Eq-p1}) implies that $SEI_{a}(T) >  SEI_{a}(T')$ for $0<a<1$.\\
Again, by virtue of Lagrange's mean value theorem there exist numbers $\Theta_{3}$, $\Theta_{4}$ such that $d_v -1 < \Theta_{3} < d_v \le d_u  < \Theta_{4} < d_u +1$ and
\begin{eqnarray*}
R_{\alpha}^{0}(T)-R_{\alpha}^{0}(T')&=& (d_v) ^{\alpha} - (d_v - 1 )^{\alpha} - [(d_u +1)^{\alpha} - (d_u) ^{\alpha}]\\
&=&\alpha (\Theta_{3}^{\alpha-1} - \Theta_{4}^{\alpha-1})\\
&&\begin{cases}
<       0 & \text{if $\alpha<0$ or $\alpha>1$,}\\
>       0 & \text{if $0<\alpha<1$.}
\end{cases}
\end{eqnarray*}
This completes the proof.

\end{proof}

If $V(G)=\{v_1,v_2,\cdots, v_n\}$ such that $d_{v_1}\ge d_{v_2}\ge \cdots \ge d_{v_n}$ then the sequence $\pi = \left(d_{v_1}, d_{v_2},\cdots,d_{v_n}\right)$ is called degree sequence of $G$.

\begin{thm}\label{thm-p1}
If $T\in\mathbb{PT}_{n,n_1}$ then $SEI_{a}(T)\ge 2a^{2}n + (a^{n_1} -2a^2 +a)n_1 - 2a^{2}$ for $0<a<1$
and
\[
R_{\alpha}^{0}(T) \begin{cases}
\le      2^{\alpha}n+(n_{1})^{\alpha}-(2^{\alpha}-1)n_{1}-2^{\alpha} & \text{if $\alpha<0$ or $\alpha>1$,}\\
\ge      2^{\alpha}n+(n_{1})^{\alpha}-(2^{\alpha}-1)n_{1}-2^{\alpha} & \text{if $0<\alpha<1$.}
\end{cases}
\]
The equality sign in any of the above inequalities holds if and only if $T$ has the degree sequence $(n_1,\underbrace{2,...,2}_{n-n_1-1},\underbrace{1,...,1}_{n_1} \ )$.

\end{thm}

\begin{proof}
The result directly follows from Lemma \ref{lem-p1}.
\end{proof}

\begin{lem}\label{lem-p2}
If $T\in\mathbb{PT}_{n,n_1}$ contains two non-pendent vertices $u,v$ such that $d_u\ge d_v +2$ then there exist $T'\in\mathbb{PT}_{n,n_1}$ such that $SEI_{a}(T) <  SEI_{a}(T')$ for $0<a<1$ and
\[
R_{\alpha}^{0}(T)
\begin{cases}
>       R_{\alpha}^{0}(T') & \text{if $\alpha<0$ or $\alpha>1$,}\\
<       R_{\alpha}^{0}(T') & \text{if $0<\alpha<1$}
\end{cases}
\]

\end{lem}

\begin{proof}
Let $w$ be the neighbor of $u$ which does not lie on the unique $u-v$ path. If $T' = T - uw + vw$ then there exist numbers $\Theta_{1}$, $\Theta_{2}$ such that $d_v  < \Theta_{1} < d_v +1 \le d_u -1 < \Theta_{2} < d_u$ and
\begin{eqnarray}\label{Eq-p1}
SEI_{a}(T)- SEI_{a}(T')&=& d_u a^{d_u} - (d_u - 1 ) a^{d_u -1} - [(d_v +1) a^{d_v +1} - d_v  a^{d_v}]\nonumber\\
&=& a^{\Theta_{2}}(1+\Theta_{2}\ln a)-a^{\Theta_{1}}(1+\Theta_{1}\ln a)<0.
\end{eqnarray}
There also exist numbers $\Theta_{3}$, $\Theta_{4}$ such that $d_v  < \Theta_{3} < d_v +1 \le d_u -1 < \Theta_{4} < d_u$  and
\begin{eqnarray*}
R_{\alpha}^{0}(T)-R_{\alpha}^{0}(T')&=& (d_u) ^{\alpha} - (d_u - 1 )^{\alpha} - [(d_v +1)^{\alpha} - (d_v) ^{\alpha}]\\
&=&\alpha (\Theta_{4}^{\alpha-1} - \Theta_{3}^{\alpha-1})\\
&&\begin{cases}
>       0 & \text{if $\alpha<0$ or $\alpha>1$,}\\
<       0 & \text{if $0<\alpha<1$.}
\end{cases}
\end{eqnarray*}
This completes the proof.

\end{proof}

\begin{lem}\label{lem-p3}
\cite{Gutman15}
If $T\in\mathbb{PT}_{n,n_1}$ such that the inequality $|d_u-d_v|\le1$ holds for all non-pendent vertices $u,v\in V(T)$, then
$n_t= (n-n_1)t - n_1 +2$
and
$n_{t+1}= n-(n-n_1)t - 2$ where $t= \left\lfloor \frac{n-2}{n-n_1} \right\rfloor + 1$.
\end{lem}

\begin{thm}\label{thm-p2}
If $T\in\mathbb{PT}_{n,n_1}$ and $t= \left\lfloor \frac{n-2}{n-n_1} \right\rfloor + 1$ then
\small
$$
SEI_{a}(T)\le [(n-n_1)t - n_1 +2] t a^{t} + [n-(n-n_1)t - 2] (t+1)a^{t+1}+ n_1 a \quad \text{for} \quad 0<a<1
$$
\normalsize
and
\small
\[
R_{\alpha}^{0}(T) \begin{cases}
\ge      [(n-n_1)t - n_1 +2] t^{\alpha} + [n-(n-n_1)t - 2] (t+1)^{\alpha}+ n_1 & \text{if $\alpha<0$ or $\alpha>1$,}\\
\le      [(n-n_1)t - n_1 +2] t^{\alpha} + [n-(n-n_1)t - 2] (t+1)^{\alpha}+ n_1 & \text{if $0<\alpha<1$.}
\end{cases}
\]
\normalsize
The equality sign in any of the above inequalities holds if and only if $T$ has the degree sequence $(\ \underbrace{t+1,\cdots,t+1}_{n-(n-n_1)t - 2},\underbrace{t,\cdots,t}_{(n-n_1)t - n_1 +2}, \underbrace{1,\cdots,1}_{n_1 }\ )$.

\end{thm}

\begin{proof}
From Lemma \ref{lem-p2} and Lemma \ref{lem-p3}, the desired result follows.
\end{proof}

\section{Zeroth-order general Randi\'{c} index, variable sum exdeg index and branching vertices of trees}

For $n\ge6$, let $\mathbb{BT}_{n,b}$ be the collection of all $n$-vertex trees with branching vertices $b$. It is known that $b\le \frac{n}{2}-1$ \cite{Lin-14}. Throughout this section we take $1\le b\le \frac{n}{2}-1$ because the set $\mathbb{BT}_{n,0}$ contains only one graph, namely the path graph $P_n$.

\begin{lem}\label{lem-b1}
If $T\in\mathbb{BT}_{n,b}$ contains a vertex having degree greater than 3 then there is a tree $T'\in\mathbb{BT}_{n,b}$ such that
\[
SEI_{a}(T)
\begin{cases}
>       SEI_{a}(T') & \text{if $a>1$,}\\
<       SEI_{a}(T') & \text{if $0<a<1$.}
\end{cases}
\]
and
\[
R_{\alpha}^{0}(T)
\begin{cases}
>       R_{\alpha}^{0}(T') & \text{if $\alpha<0$ or $\alpha>1$,}\\
<       R_{\alpha}^{0}(T') & \text{if $0<\alpha<1$}
\end{cases}
\]

\end{lem}

\begin{proof}
Let $u\in V(T)$ be a vertex having degree greater than 3. Let $P = v_0 v_1 \cdots v_{r+1}$ be a longest path in $T$ containing $u$, where $u=v_i$ for some $i\in\{1,2,\cdots,r\}$. Let $w$ be a neighbor of $u$ different from both $v_{i-1},v_{i+1}$. If $T' = T - uw + wv_{r+1}$ then
\begin{eqnarray*}
SEI_{a}(T)- SEI_{a}(T')&=&d_u a^{d_u} - (d_u - 1 ) a^{d_u -1} - (2a^2 -a)\\
&=& a^{\Theta_{2}}(1+\Theta_{2}\ln a)-a^{\Theta_{1}}(1+\Theta_{1}\ln a)\\
&&\begin{cases}
>       0 & \text{if $a>1$,}\\
<       0 & \text{if $0<a<1$,}
\end{cases}
\end{eqnarray*}
where $1< \Theta_{1} < 2 < d_u -1 < \Theta_{2} < d_u$. Also, we have
\begin{eqnarray*}
R_{\alpha}^{0}(T)-R_{\alpha}^{0}(T')&=& (d_u) ^{\alpha} - (d_u - 1 )^{\alpha} - (2^\alpha - 1)\\
&&\begin{cases}
>       0 & \text{if $\alpha<0$ or $\alpha>1$,}\\
<       0 & \text{if $0<\alpha<1$.}
\end{cases}
\end{eqnarray*}

\end{proof}

\begin{lem}\label{lem-b2}
\cite{Bor15} If $T\in\mathbb{BT}_{n,b}$ has maximum degree 3 then $n_1= b+2$ and $n_2 = n-2b-2$.

\end{lem}

\begin{thm}\label{thm-b1}
If $T\in\mathbb{BT}_{n,b}$ then
\[
SEI_{a}(T)
\begin{cases}
\ge       2a^2n + (3a^3 -4a^2+a)b -2a(2a -1) & \text{if $a>1$,}\\
\le       2a^2n + (3a^3 -4a^2+a)b -2a(2a -1) & \text{if $0<a<1$}
\end{cases}
\]
and
\[
R_{\alpha}^{0}(T)
\begin{cases}
\ge      2^\alpha n + (3^\alpha -2^{\alpha+1}+1)b -2^{\alpha+1} +2 & \text{if $\alpha<0$ or $\alpha>1$,}\\
\le      2^\alpha n + (3^\alpha -2^{\alpha+1}+1)b -2^{\alpha+1} +2 & \text{if $0<\alpha<1$.}
\end{cases}
\]
The equality sign in any of the above inequalities holds if and only if $T$ has the degree sequence $(\ \underbrace{3,...,3}_{b},\underbrace{2,...,2}_{n-2b-2},\underbrace{1,...,1}_{b+2} \ )$.
\end{thm}

\begin{proof}
The result follows from Lemma \ref{lem-b1} and Lemma \ref{lem-b2}.
\end{proof}

\begin{lem}\label{lem-b3}
If $T\in\mathbb{BT}_{n,b}$ contains two or more vertices having degree greater than 3 then there exist $T'\in\mathbb{BT}_{n,b}$ such that
\[
R_{\alpha}^{0}(T)
\begin{cases}
<       R_{\alpha}^{0}(T') & \text{if $\alpha<0$ or $\alpha>1$,}\\
>       R_{\alpha}^{0}(T') & \text{if $0<\alpha<1$}
\end{cases}
\]
and
\[
SEI_{a}(T)
\begin{cases}
<       SEI_{a}(T') & \text{if $a>1$,}\\
>       SEI_{a}(T') & \text{if $0<a<1$.}
\end{cases}
\]
\end{lem}

\begin{proof}
Let $u,v\in V(T)$ such that $d_u\ge d_v\ge4$. Suppose $N_T (v) = \{ v_1, v_2, \cdots, v_{r-1},v_{r}\}$ and let $u$ be connected to $v$ through $v_{r}$ (it is possible that $u=v_{r}$). If $T' = T - \{ vv_1, vv_2, \cdots, vv_{r-3}\} + \{ uv_1, uv_2, \cdots, uv_{r-3}\}$ then
\begin{eqnarray*}
R_{\alpha}^{0}(T)-R_{\alpha}^{0}(T')&=& (d_v)^{\alpha} - 3^\alpha - [(d_u +d_v -3)^{\alpha} - (d_u)^{\alpha}]\\
&=& \alpha (d_v -3)(\Theta_{1} ^ {\alpha-1}-\Theta_{2} ^ {\alpha-1})\\
&&\begin{cases}
<       0 & \text{if $\alpha<0$ or $\alpha>1$,}\\
>       0 & \text{if $0<\alpha<1$.}
\end{cases}
\end{eqnarray*}
where $3< \Theta_{1} < d_v \le d_u < \Theta_{2} < d_u + d_v -3$. Also, we have
\begin{eqnarray*}
SEI_{a}(T)- SEI_{a}(T')&=& d_v  a^{d_v} - 3 a^{3} - [(d_u +d_v -3) a^{d_u +d_v -3} - d_u  a^{d_u}]\\
&=& (d_v-3)\left[a^{\Theta_{3}}(1+\Theta_{3}\ln a)-a^{\Theta_{4}}(1+\Theta_{4}\ln a)\right]\\
&&\begin{cases}
<       0 & \text{if $a>1$,}\\
>       0 & \text{if $0<a<1$,}
\end{cases}
\end{eqnarray*}
where $3< \Theta_{3} < d_v \le d_u < \Theta_{4} < d_u + d_v -3$.
\end{proof}

\begin{lem}\label{lem-b4}
If $T\in\mathbb{BT}_{n,b}$ contains at least one vertex of degree 2 then there is $T'\in\mathbb{BT}_{n,b}$ such that
\[
R_{\alpha}^{0}(T)
\begin{cases}
<       R_{\alpha}^{0}(T') & \text{if $\alpha<0$ or $\alpha>1$,}\\
>       R_{\alpha}^{0}(T') & \text{if $0<\alpha<1$}
\end{cases}
\]
and
\[
SEI_{a}(T)
\begin{cases}
<       SEI_{a}(T') & \text{if $a>1$,}\\
>       SEI_{a}(T') & \text{if $0<a<1$.}
\end{cases}
\]
\end{lem}

\begin{proof}
The assumption $b\ge1$ implies that there exist two adjacent vertices $u,v\in V(T)$ such that $d_u\ge 3$ and $d_v=2$. Let $N_T (v) = \{ u,w\}$ and $T' = T - vw + uw$. Now, the desired result easily follows by observing the differences $R_{\alpha}^{0}(T)-R_{\alpha}^{0}(T')$ and $SEI_{a}(T)- SEI_{a}(T')$.

\end{proof}

\begin{lem}\label{lem-b5}
\cite{Bor15} If $T\in\mathbb{BT}_{n,b}$ has no vertex of degree 2 and has at most one vertex of degree greater than 3 then $T$ has the degree sequence $(n-2b+1, \underbrace{3,...,3}_{b-1},\underbrace{1,...,1}_{n-b} \ )$.

\end{lem}

\begin{thm}\label{thm-b2}
If $T\in\mathbb{BT}_{n,b}$ then
\[
R_{\alpha}^{0}(T)
\begin{cases}
\le      (n-2b+1)^\alpha + n + (3^\alpha - 1)b -3^\alpha & \text{if $\alpha<0$ or $\alpha>1$,}\\
\ge      (n-2b+1)^\alpha + n + (3^\alpha - 1)b -3^\alpha & \text{if $0<\alpha<1$}
\end{cases}
\]
and
\[
SEI_{a}(T)
\begin{cases}
\le       (n-2b+1)a^{n-2b+1} + na + (3a^3 - a)b -3a^3 & \text{if $a>1$,}\\
\ge       (n-2b+1)a^{n-2b+1} + na + (3a^3 - a)b -3a^3 & \text{if $0<a<1$.}
\end{cases}
\]
The equality sign in any of the above inequalities holds if and only if $T$ has the degree sequence $(n-2b+1, \underbrace{3,...,3}_{b-1},\underbrace{1,...,1}_{n-b} \ )$.
\end{thm}

\begin{proof}
The result follows from Lemma \ref{lem-b3}, Lemma \ref{lem-b4} and Lemma \ref{lem-b5}.

\end{proof}

\section{Zeroth-order general Randi\'{c} index, variable sum exdeg index and segments of trees}

 For $n\ge6$, denote by $\mathbb{ST}_{n,k}$ the set of all $n$-vertex trees with $k$ segments. Throughout this section we take $3\leq k\leq n-2$ because
 $\mathbb{ST}_{n,1}= \{P_n\}$, $\mathbb{ST}_{n,n-1}=\{S_n\}$ and the set $\mathbb{ST}_{n,2}$ is empty.

Squeeze of an $n$-vertex tree $T$ (is denoted by $S(T)$) is a tree obtained from $T$ by replacing each segment with an edge \cite{12}. Hence
\begin{equation}\label{Eq2}
k=|E(S(T))| = |V(S(T))| - 1 = n - n_2 - 1
\end{equation}
By an even-prime vertex we mean a vertex with degree 2. From Equation \ref{Eq2}, it is clear that the problem of finding extremal trees from the collection $\mathbb{ST}_{n,k}$ is equivalent to the problem of finding extremal trees from the collection of all $n$-vertex trees with fixed even-prime vertices.

\begin{lem}\label{lem-0.2}
\cite{v-ukicevic-11b} If $T$ is an $n$-vertex tree then
\[
SEI_{a}(T)\le(n-1)a^{n-1}+(n-1)a
\]
for $a>1$ and $n\ge4$. The equality sign in the inequality holds if and only if $T\cong S_n$.
\end{lem}

\begin{lem}\label{lem-0.1}
\cite{9} For $n\ge4$, if $T$ is an $n$-vertex tree then
\[
R_{\alpha}^{0}(T) \begin{cases}
\le       (n-1)^{\alpha}+(n-1) & \text{if $\alpha<0$ or $\alpha>1$,}\\
\ge       (n-1)^{\alpha}+(n-1) & \text{if $0<\alpha<1$.}
\end{cases}
\]
The equality sign in the inequality holds if and only if $T\cong S_n$.
\end{lem}

\begin{thm}\label{lem-0.3}
If $T\in \mathbb{ST}_{n,k}$ then\\
\[
R_{\alpha}^{0}(T) \begin{cases}
\le      2^{\alpha}n+k^{\alpha}-(2^{\alpha}-1)k-2^{\alpha} & \text{if $\alpha<0$ or $\alpha>1$,}\\
\ge      2^{\alpha}n+k^{\alpha}-(2^{\alpha}-1)k-2^{\alpha} & \text{if $0<\alpha<1$,}
\end{cases}
\]
and
\[
SEI_{a}(T)\le 2a^{2}n+ka^{k}-(2a-1)ak-2a^{2}
\]
for $a>1$. The equality sign in any of the above inequalities holds if and only if $T$ has the degree sequence $(k,\underbrace{2,...,2}_{n-k-1},\underbrace{1,...,1}_{k} \ )$.
\end{thm}

\begin{proof}

By definition of the squeeze of a tree, zeroth-order general Randi\'{c} index and variable sum exdeg index, we have
\begin{equation}\label{Eq-3}
R_{\alpha}^{0}(T)=R_{\alpha}^{0}(S(T))+2^{\alpha}n_2
\end{equation}
and
\begin{equation}\label{Eq-4}
SEI_{a}(T)=SEI_{a}(S(T))+2a^{2}n_2.
\end{equation}
From Equation (\ref{Eq2}), we have $n_2=n-k-1$ and hence from
Equation (\ref{Eq-3}) and Equation (\ref{Eq-4}) it follow that
\begin{equation}\label{Eq-6}
R_{\alpha}^{0}(T)=R_{\alpha}^{0}(S(T))+2^{\alpha}(n-k-1)
\end{equation}
and
\begin{equation}\label{Eq-7}
SEI_{a}(T)=SEI_{a}(S(T))+2a^{2}(n-k-1).
\end{equation}
Since $S(T)$ has $n-n_2=k+1$ vertices. So, by Lemma \ref{lem-0.2} and Lemma \ref{lem-0.1}, we have
\[
SEI_{a}(S(T))\le ka^{k}+ka
\quad
\text{and}
\quad
R_{\alpha}^{0}(S(T)) \begin{cases}
\le       k^{\alpha}+k & \text{if $\alpha<0$ or $\alpha>1$,}\\
\ge       k^{\alpha}+k & \text{if $0<\alpha<1$,}
\end{cases}
\]
where $a>1$ and the equality sign in any of the above inequalities holds if and only if $S(T)\cong S_{k+1}$.
Now, from Equation (\ref{Eq-6}) and Equation (\ref{Eq-7}) the desired result follows.

\end{proof}

A caterpillar is a tree which results in a path graph by deletion of all pendent vertices and incident edges.

\begin{lem}\label{lem-1}
\cite{12}
 If $T$ is an $n$-vertex non-caterpillar then there exist an $n$-vertex caterpillar $T^{'}$ such that $T^{'}$ and $T$ have the same degree sequence (and same number of segments).
\end{lem}

\begin{lem}\label{lem-1a}
If $T\in\mathbb{ST}_{n,k}$ has maximum degree greater than 4 then there exist $T'\in\mathbb{ST}_{n,k}$ such that
 \[
R_{\alpha}^{0}(T)
\begin{cases}
>       R_{\alpha}^{0}(T') & \text{if $\alpha<0$ or $\alpha>1$,}\\
<       R_{\alpha}^{0}(T') & \text{if $0<\alpha<1$}
\end{cases}
\]
and
\[
SEI_{a}(T)
\begin{cases}
>       SEI_{a}(T') & \text{if $a>1$,}\\
<       SEI_{a}(T) & \text{if $0<a<1$.}
\end{cases}
\]

\end{lem}

\begin{proof}
Let $\pi$ be the degree sequence of the tree $T$. By Lemma \ref{lem-1}, there must exist a caterpillar $T^{(1)}\in \mathbb{ST}_{n,k}$ with degree sequence $\pi$ (it is possible that $T=T^{(1)}$). Obviously,
\[
R_{\alpha}^{0}\left( T \right)= R_{\alpha}^{0}\left( T^{(1)} \right)  \quad \text{and} \quad SEI_{a}\left( T \right)=SEI_{a}\left( T^{(1)} \right).
\]
Let $P: v_{0}v_{1}\ldots v_{r}v_{r+1}$ be the longest path in $T^{(1)}$ containing the vertex of degree greater than 4. Obviously, $v_0$ and $v_{r+1}$ are pendent vertices. Let $d_{v_i}\ge5$ for some $i\in\{1,2,\ldots,r\}$. The assumption that $T^{(1)}$ is a caterpillar implies that there exist two pendent vertices $u_{1}, u_{2}$ adjacent to $v_i$, not included in the path $P$. Let $T'=T^{(1)} - \{u_{1}v_i, u_{2}v_i\} + \{u_{1}v_{r+1},u_{2}v_{r+1}\}$. Clearly, $T'\in \mathbb{ST}_{n,k}$. By virtue of Lagrange's mean value theorem there exists numbers $\Theta_{1}$, $\Theta_{2}$ such that $1<\Theta_{1}<3 \le d_{v_i}-2<\Theta_{2}<d_{v_i}$ and
\begin{eqnarray*}
R_{\alpha}^{0}\left( T \right) - R_{\alpha}^{0}\left( T'\right)=R_{\alpha}^{0}\left(T^{(1)}\right)-R_{\alpha}^{0}(T^{'})
&=& [\left(d_{v_i}\right)^{\alpha}-(d_{v_i}-2)^{\alpha}]-[3^{\alpha}-1^{\alpha}]\\
&=& 2\alpha(\Theta_{2}^{\alpha-1}-\Theta_{1}^{\alpha-1})\\
&& \begin{cases}
>       0 & \text{if $\alpha<0$ or $\alpha>1$,}\\
<       0 & \text{if $0<\alpha<1$.}
\end{cases}
\end{eqnarray*}

Also, there exists numbers $\Theta_{3}$, $\Theta_{4}$ such that $1<\Theta_{3}<3 \le d_{v_i}-2<\Theta_{4}< d_{v_i}$ and
 \begin{eqnarray*}
SEI_{a}\left( T \right)-SEI_{a}\left( T' \right) &=& SEI_{a}\left(T^{(1)}\right)-SEI_{a}(T')\\
&=& \left[d_{v_i}a^{d_{v_i}}-(d_{v_i}-2)a^{(d_{v_i}-2)}\right]-[3a^{3}-a]\\
&=& 2a^{\Theta_{4}}(1+\Theta_{4}\ln a)-2a^{\Theta_{3}}(1+\Theta_{3}\ln a)\\
&& \begin{cases}
>       0 & \text{if $a>1$,}\\
<       0 & \text{if $0<a<1$.}
\end{cases}
\end{eqnarray*}

\end{proof}

\begin{lem}\label{lem-1aa}
If $T\in\mathbb{ST}_{n,k}$ has two or more vertices of degree 4 then there exist $T'\in\mathbb{ST}_{n,k}$ such that
 \[
R_{\alpha}^{0}(T)
\begin{cases}
>       R_{\alpha}^{0}(T') & \text{if $\alpha<0$ or $\alpha>1$,}\\
<       R_{\alpha}^{0}(T') & \text{if $0<\alpha<1$}
\end{cases}
\]
and
\[
SEI_{a}(T)
\begin{cases}
>       SEI_{a}(T') & \text{if $a>1$,}\\
<       SEI_{a}(T) & \text{if $\frac{1+\sqrt{33}}{16}<a<1$.}
\end{cases}
\]

\end{lem}

\begin{proof}
Let $\pi$ be degree sequence of the tree $T$. By Lemma \ref{lem-1}, there must exist a caterpillar $T^{(1)}\in \mathbb{ST}_{n,k}$ with degree sequence $\pi$ (it is possible that $T=T^{(1)}$). Obviously,
\[
R_{\alpha}^{0}\left( T \right)= R_{\alpha}^{0}\left( T^{(1)} \right)  \quad \text{and} \quad SEI_{a}\left( T \right)=SEI_{a}\left( T^{(1)} \right).
\]
Suppose that the vertices $u,v\in V(T^{(1)})$ have degree 4. Let $P: v_{0}v_{1}\ldots v_{r}v_{r+1}$ be the longest path in $T^{(1)}$ containing the vertices $u,v$. Let $u={v_i}$ and $u={v_j}$ for some $i,j\in\{1,2,\ldots,r\}$, $i\neq j$. There must exist two pendent vertices $u_{1}, u_{2}$, not included in the path $P$ such that $u_1v_i, u_2v_j\in E(T^{(1)})$. Let $T'=T^{(1)} - \{u_1v_i, u_2v_j\} + \{u_{1}v_{r+1},u_{2}v_{r+1}\}$. Clearly, $T'\in \mathbb{ST}_{n,k}$ and
\begin{eqnarray*}
R_{\alpha}^{0}\left( T \right) - R_{\alpha}^{0}\left( T'\right)=R_{\alpha}^{0}\left(T^{(1)}\right)-R_{\alpha}^{0}(T^{'})
&=& 2(4^{\alpha}-3^{\alpha})-(3^{\alpha}-1)\\
&& \begin{cases}
>       0 & \text{if $\alpha<0$ or $\alpha>1$,}\\
<       0 & \text{if $0<\alpha<1$.}
\end{cases}
\end{eqnarray*}
Also, we have
\begin{eqnarray*}
SEI_{a}\left( T \right)-SEI_{a}\left( T' \right) = SEI_{a}\left(T^{(1)}\right)-SEI_{a}(T')
&=& a(8a^{3}-9a^{2}+1)\\
&& \begin{cases}
>       0 & \text{if $a>1$,}\\
<       0 & \text{if $\frac{1+\sqrt{33}}{16}<a<1$.}
\end{cases}
\end{eqnarray*}

\end{proof}



\begin{lem}\label{lem-4}
\cite{12} If $T$ is a tree satisfying $\Delta\leq4$ and $n_{4}\leq1$ then the degree sequence $\pi$ of $T$ is given below
\[
\pi = \begin{cases}
       ( 4, \underbrace{3,...,3}_{\frac{k-4}{2}},\underbrace{2,...,2}_{n-k-1},\underbrace{1,...,1}_{\frac{k+4}{2}} \ ) & \text{if $k$ is even,}\\
        ( \ \underbrace{3,...,3}_{\frac{k-1}{2}},\underbrace{2,...,2}_{n-k-1},\underbrace{1,...,1}_{\frac{k+3}{2}} \ ) & \text{if $k$ is odd.}
\end{cases}
\]

\end{lem}

\begin{thm} Let $T\in \mathbb{ST}_{n,k}$ where $3\leq k\leq n-2$.\\ \\
\textbf{(i).} If $\alpha< 0$ or $\alpha> 1$, then the following inequality holds:
\begin{eqnarray*}
R_{\alpha}^{0}(T) &\geq&
\begin{cases}
       f(n,k)+ 4^{\alpha}-2\cdot3^{\alpha}-2^{\alpha}+2 & \text{if $k$ is even},\\
       f(n,k)+ \frac{3-3^{\alpha}-2^{\alpha+1}}{2} & \text{if $k$ is odd},
\end{cases}
\end{eqnarray*}
where $f(n,k)=2^{\alpha}n+\left(\frac{3^{\alpha}-2^{\alpha+1}+1}{2}\right)k$. If $0<\alpha<1$ then the inequality is reversed.\\ \\
\textbf{(ii).} For $a> 1$, the following inequality holds:
\begin{eqnarray*}
SEI_{a}(T)&\geq&
\begin{cases}
       g(n,k)+ 4a^{4}-6a^{3}-2a^{2}+2a & \text{if $k$ is even},\\
       g(n,k)+\frac{3a-3a^{3}-4a^{2}}{2} & \text{if $k$ is odd},
\end{cases}
\end{eqnarray*}
where $g(n,k)=2a^{2}n+\left(\frac{3a^{3}-4a^{2}+a}{2}\right)k$. If $\frac{1+\sqrt{33}}{16}<a<1$ then the inequality is reversed.\\

In each part, the bound is best possible and is attained if and only if $T$ has the degree sequence $\pi$ given below:
\[
\pi = \begin{cases}
       ( 4, \underbrace{3,...,3}_{\frac{k-4}{2}},\underbrace{2,...,2}_{n-k-1},\underbrace{1,...,1}_{\frac{k+4}{2}} \ ) & \text{if $k$ is even,}\\
        ( \ \underbrace{3,...,3}_{\frac{k-1}{2}},\underbrace{2,...,2}_{n-k-1},\underbrace{1,...,1}_{\frac{k+3}{2}} \ ) & \text{if $k$ is odd.}
\end{cases}
\]
\end{thm}

\begin{proof}

From Lemma \ref{lem-1a}, Lemma \ref{lem-1aa} and Lemma \ref{lem-4}, the desired result follows.

\end{proof}

\end{document}